\documentclass[11pt]{article}
\usepackage{amsmath}
\usepackage{amsthm}
\usepackage{amssymb}

\newtheorem{theo}{Theorem}[section]

\newtheorem{prop}[theo]{Proposition}
\newtheorem{lemma}[theo]{Lemma}
\newtheorem{coro}[theo]{Corollary}

\newcommand{\CC}{{\cal C}}

\newcommand{\eps}{{\varepsilon}}

\begin{document}
\date{}

\title{Erasure list-decodable codes and Tur\'an hypercube problems}

\author{Noga Alon
\thanks{Princeton University,
Princeton, NJ, USA 
and
Tel Aviv University, Tel Aviv,
Israel.
Email: {\tt nalon@math.princeton.edu}.
Research supported in part by
NSF grant DMS-2154082.}
}

\maketitle
\begin{abstract}
We observe that several vertex Tur\'an type problems for the hypercube
that received a considerable amount of attention in the combinatorial
community are equivalent to questions about erasure list-decodable
codes. Analyzing a recent construction of Ellis, Ivan and Leader,
	and determining the Tur\'an density of certain hypergraph
	augemntations we
obtain improved bounds for some of these problems.
\end{abstract}
\section{Introduction and results}

\subsection{Erasure codes and Tur\'an hypercube problems}

A set $\CC$ of binary vectors of 
length $n$ is a $(d,L)$-list decodable erasure
code of length $n$ (a $(d,L,n)$-code, for short) if  
for every codeword $w$, after erasing any $d$-bits of $w$,
the remaining part of the vector has at most $L$ possible 
completions into codewords of $\CC$. Erasure list-decodable codes
are considered in
\cite{Gu}, see also \cite{BDT} and the references therein. These papers
deal with codes of rate smaller than $1$, that is, the cardinality
of $\CC$ is exponentially smaller than $2^n$. 

Here we consider much denser codes, where the cardinality of $\CC$ is
a constant fraction of all $2^n$ vectors. This range of the parameters
is not very natural from the information theoretic point of view,
but it is equivalent to a problem that received a considerable amount of
attention in the combinatorial community, see \cite{Ko}, \cite{JE},
\cite{AKS}, \cite{HN}, \cite{BLM}, \cite{JT}, \cite{EK}, \cite{EIL}. 
Indeed, $\CC$ is a $(d,L,n)$-code
if and only if it is a subset of vertices of the
discrete $n$-cube $Q_n$ that contains at most $L$ vertices of any 
$d$-dimensional subcube of $Q_n$. In this language, for example, the result of
\cite{Ko}, proved independently in \cite{JE}, is that the maximum possible
cardinality of a $(2,3,n)$-code is $\lceil 2^{n+1}/3 \rceil$. 

\subsection{$(d,2^d-1,n)$-codes}

An intriguing special case of the general problem of determining or
estimating the maximum possible cardinality of $(d,L,n)$-codes is the 
cases $L=2^d-1$ corresponding to codes $\CC$ that contain no full copy of
a $d$-subcube. Here it is more natural to consider the complement
and denote by $g(n,d)$ the smallest cardinality of a subset of
the vertices intersecting
every $d$-subcube. Let $\gamma_d$ denote the limit
$\lim_{n \mapsto \infty} g(n,d)/2^n$ 
(it is easy to see that the limit exists
as for any fixed $d$, 
$g(n,d)/2^n$ is a monotone increasing function of $n$). Trivially,
$\gamma_1=1/2$, and the result of \cite{Ko} and \cite{JE} mentioned above 
is that
$\gamma_2=1/3$.  In \cite{AKS} it is shown that 
$\gamma_d \geq \log_2(d+2)/2^{d+2}$. It has been a folklore
conjecture (see \cite{BLM}) that $\gamma_d=1/(d+1)$  but this is refuted
in a very strong sense in a recent paper 
of Ellis, Ivan and Leader  \cite{EIL}, where it is shown that 
$(\gamma_d)^{1/d} \leq 2^{-1/8+o(1)})$. As mentioned in
\cite{EIL} we observed that their argument
can be improved to show that $(\gamma_d)^{1/d} \leq 2^{-1/2+o(1)}$.  
This is stated in the following proposition.
\begin{prop}
\label{p11}
For every large $k$ and every $n$ there is a subset of less than
a fraction of $2^{-k}$ of the vertices of the $n$-cube that intersects
the set of vertices of any cube of dimension $d=2k + 3 \log_2 k$.
\end{prop}

\subsection{Codes of positive density}

Another range of the parameters of $(d,L,n)$-codes that has been studied
in quite a few combinatorial papers deals with the minimum possible
$L=L(d)$ so that there exists infinitely many $(d,L,n)$-codes of positive
density. More precisely, let $L(d)$ denote the
smallest possible $L$ so that there exists an $\eps=\eps(d)>0$ such
that for every $n$ there is a $(d,L,n)$-code of cardinality at least
$\eps 2^n$. The problem of determining or estimating $d(L)$ is considered
in \cite{JT} (where it is denoted by $\mu(d)$.)  A conjecture suggested
in \cite{BLM} asserts that $L(d)= {d \choose {\lfloor d/2 \rfloor}}$.
It is easy to see that this is always an upper bound for $L(d)$, 
and that it 
holds as equality for $d \leq 3$. However, somewhat surprisingly
this conjecture too is refuted by the recent construction of \cite{EIL}
which shows that $L(d)$ is at most $(5/6){d \choose {\lfloor d/2 \rfloor}}$
for every $d \geq 4$.  The authors of \cite{JT} proved a lower bound
for $L(d)$, showing that it is at least $t_2(d)$ + $t_3(d)$, where
$t_2(d)$ is $0$ if $\lceil d/3 \rceil$ is odd and $1$ otherwise,
and $t_3(d)$ is $3^{d/3}$ for $d \equiv 0 \bmod 3$, is
$4 \cdot 3^{(d-4)/3}$ for $d \equiv 1 \bmod 3$ and is
$2 \cdot  3^{(d-2)/3}$ for $3 \equiv 2 \bmod 3$. In particular, this 
shows that 
$$
L(5) \geq 7, L(6) \geq 10, L(7) \geq 12,
L(8) \geq 18, L(9) \geq 27, L(10) \geq 37.
$$
Here we improve the lower bounds for all $d \geq 5$, proving,
in particular, the following
\begin{prop}
\label{p12}
$$
L(5) \geq 8, L(6) \geq 12, L(7) \geq 20, L(8) \geq 32, L(9) \geq 48, 
L(10) \geq  80.
$$
\end{prop}
For large $d$ we prove that 
$L(d) \geq d \cdot 3^{(d-6)/3}$ for $d$ divisible
by $3$ and obtain a similar bound for $d$ that is not divisible by $3$.
The improved lower bounds are obtained by applying the simple result about
graph and hypergraph augmentations described in the following subsection.

We also improve the upper bounds as follows:
\begin{theo}
\label{t13}
$$
L(5)=8, L(6) \leq 16, L(7) \leq 28
$$
and for large $d$, 
$$
	L(d) \leq (c+o(1)) {d \choose {\lfloor d/2 \rfloor}}
$$
where  
$$
	c= \lim_{t \mapsto \infty} \frac{(2^t-2)(2^t-4) \ldots (2^t-2^{t-1})}
	{(2^t-1)^{t-1}}
$$
is roughly $0.29$.
\end{theo}
Note that by the results above the exact values of
$L(d)$ for $1 \leq d \leq 5$ are given by the sequence $1,2,3,5,8$.

\subsection{Graph and hypergraph augmentations}

For a graph $G=(V,E)$ and an integer $r \geq 2$, let the $r$-augmentation
of $G$, denoted by $G(r)$, be the  $r$-uniform hypergraph 
$(V \cup S, \{e \cup S: e \in E\})$, where $S \cap V = \emptyset$ and
$|S|=r-2$. Thus $G(r)$ is obtained from $G$ by adding the same set of
$r-2$ vertices to each edge of $G$. This set is called the stem of $G(r)$.
More generally, for a $k$-uniform hypergraph $H=(V,E)$ and an integer
$r \geq k$, let the $r$-augmentation of $H$, denoted $H(r)$, be the 
$r$-uniform hypergraph 
$(V \cup S, \{e \cup S: e \in E\})$, where $S \cap V = \emptyset$ and
$|S|=r-k$. 

For a fixed $r$-uniform hypergraph $F$ and for an integer $n$ let
$ex(n,F)$ denote the maximum possible number of edges in an
$r$-uniform hypergraph on $n$ vertices that contains no copy of $F$. The
Tur\'an density $\pi(F)$ of $F$ is the limit, 
as $n$ tends to infinity, of the
ratio $ex(n,F)/{n \choose r}$ (it is easy to see that this limit
always exists, and lies in $[0,1]$.)

The recent construction of Ellis, Ivan and Leader in \cite{EIL} 
implies that
if the chromatic number of a graph $G$ satisfies $\chi(G) \geq 4$, then the 
Tur\'an density of $G(r)$ is at least $0.29$ for every $r$. 
(The construction
in \cite{EIL} is described for $G=K_4$, but it is not difficult 
to check that 
it works for every graph $G$ of chromatic number at least $4$).

Here we observe
that if $\chi(G) \leq 3$ then the Tur\'an density of $G(r)$ tends to 
zero as $r$ tends to infinity. This gives a full characterization of
all the fixed graphs $G$ that must appear as links in any
$r$-uniform hypergraph of positive density (with at least $2r+1$ vertices,
say), provided $r$ is sufficiently large. This result has been
proved independently by Robert Johnson \cite{Jo}
\begin{prop}
\label{p14}
For every fixed graph $G$ with chromatic number at most $3$,
the limit of the Tur\'an density of $G(r)$ as $r$ tends to infinity is
$0$.
\end{prop}
The argument easily extends to augmentations of hypergraphs, 
giving the following
\begin{prop}
\label{p15}
For any fixed 
$k$-uniform hypergraph $H$ in which the set of vertices is 
the disjoint union of $k+1$ subsets, so that every edge contains
at most one vertex in each subset, the Tur\'an density of $H(r)$
tends to $0$ as $r$ tends to infinity
\end{prop}
\vspace{0.2cm}

\noindent
{\bf Remark:}\, 
By averaging over $r$,  Proposition \ref{p14} implies that for
every fixed $\eps>0$ and every fixed graph $G$ of chromatic number
at most $3$, if $n>n_0(G,\eps)$ then
any family of at least $\eps 2^n$ subsets of $[n]=
\{1,2,\ldots ,n\}$  contains  a copy of $G(r)$ for some $r$. Here, too,
the construction in  \cite{EIL} implies that this is false for graphs
$G$ of chromatic number at least $4$. Similarly, Proposition 
\ref{p15} implies the corresponding result for the hypercube.

\section{Proofs}
\subsection{Augmentations}
\label{ss21}

In this subsection we describe the short proof of Proposition \ref{p14}.
The proof of Proposition \ref{p15} is essentially identical.

Fix an $\eps>0$, suppose $n \geq 2r+1$  and let
$H$ be an $r$-uniform hypergraph on $n$ vertices with at least
$\eps { n \choose r}$ edges. By averaging there is a subset
$U$ of $2r+1$ vertices so that $H$ contains at least
$\eps {{2r+1} \choose r}$ edges in $U$. Let $W$ be a random subset
of size $r+1$ of $U$. The expected number
of edges contained in $W$
is at least $\eps(r+1)$. If $W$ contains $k \geq 3$ edges, then
any collection of $3$ of them gives a copy of $K_3(r)$. Thus we get
${k \choose 3}$ such copies on the set of vertices $W$. By convexity
(assuming, say, $\eps(r+1) >10$)  
this implies that the total number of copies of $K_3(r)$ that are
contained in $U$ is at least 
$$
{{2r+1} \choose {r+1}} \cdot {{\eps(r+1)} \choose 3}.
$$
By averaging over the ${{2r+1} \choose {r-2}}$ possible 
stems we get that there is one common stem for at least
$$
\frac{{{2r+1} \choose {r+1}}}
{{{2r+1} \choose {r-2}}}
\cdot {{\eps(r+1)} \choose 3}> \frac{\eps^3}{2}{r \choose 3}
$$
copies of $K_3(r)$, where here we assumed, say, $r>10/\eps$. 
This gives the existence of a graph $F$ on a subset of $r+3$ 
vertices of $U$ so that $F$ contains more than
$\frac{\eps^3}{2}{r \choose 3}$ triangles and our hypergraph contains
a copy of $F(r)$. By the known results about the Tur\'an density
of $3$-uniform, $3$-partite hypergraphs first proved in \cite{Er},
for every $s$ and every 
sufficiently large $r > r_0(\eps,s)$,
$F$ contains a complete $3$-partite graph $T$ with $s$ vertices in each
vertex class. Since $F(r)$ contains $T(r)$ this completes the proof
of the proposition.  \hfill $\Box$

\subsection{Hitting subcubes}

In this subsection we describe the proof of Proposition \ref{p11}. The proof
is identical to the one in \cite{EIL} with 
one modification, replacing a naive
estimate for the maximum possible number of $k$-wise independent vectors
in $F_2^s$ by the Plotkin bound \cite{Pl}, which is a classical result in 
the theory of Error Correcting Codes. 

For simplicity we omit all floor and ceiling signs whenever these are not
crucial. All logarithms are in base $2$ unless otherwise specified.

Following the notation in \cite{EIL}, for integers $t >s$ and $r > s$,
let $D_r(s,t)$ denote the $r$-uniform hypergraph obtained by adding a stem
of size $r-s$ to every edge of the complete $s$-uniform hypergraph 
$T$ on $t$ vertices. In the notation of the previous subsection
$D_r(s,t)$ is $T(r)$. In \cite{EIL}, Theorem 6, it is proved that   
for every fixed $k$ and every (large) $r$,
the Tur\'an density of $D_r(k,8k+1)$ is at least $1-O(2^{-k})$.
The following lemma provides a quantitative improvement.
\begin{lemma}
\label{l21}
For every fixed (large) integer $k$ and every (large) $r$, 
the Tur\'an density
of $D_r(k+2 \log k, 2k+3\log k)$ is larger than $1-2^{-k}$.
\end{lemma}
\begin{proof}
Suppose $r$ is large and consider the hypergraph on a set of 
$2^{r+k}-1$ vertices indexed by the nonzero vectors in $F_2^{r+k}$,
where an $r$-set forms an edge iff it is linearly independent. It is
easy to see that the density of this hypergraph is larger than 
$1-2^{-k}$. We claim that it contains no copy of $D_r(s,t)$ where
$s=k+2 \log k, t=2k+3 \log k$. Indeed, as in the proof in \cite{EIL}, 
the existence of such a set
would give a collection of $t$ binary vectors in $F_2^{k+s}$
so that every subset of $s$ of them is linearly
independent. Let $A$ be the $k+s$ by $t$ matrix whose columns
are these $t$ vectors and consider the linear code whose parity check matrix
is $A$. This is the code consisting of all binary vectors of length $t$
that are orthogonal to every row of $A$. The dimension of this code
is at least $t-(k+s)=\log k$ and hence the number of vectors in it is
at least $k$. However, the minimum distance of this code is at least
$s+1$, since every set of $s$ columns of $A$ is linearly independent.
By the Plotkin bound it follows that the number of vectors in the code
cannot exceed $2\frac{(s+2)}{2s+2-t}<k$, contradiction.
Therefore this hypergraph contains no copy of $D_r(s,t)$. The assertion
of the lemma follows by considering blow ups of this hypergraph, which
for large $r$ hardly change the density.
\end{proof}

Returning to the proof of Proposition \ref{p11}
we apply the lemma and take the
union of the 
complement of the construction it provides 
in every
(large) layer $r$ of the hypercube. In the small layers we simply
take all vertices. This gives a set of vertices of 
the $n$-cube that contains
less than a fraction of $2^{-k}$ of the vertices and intersects every 
copy of $D_r(k+2 \log k, 2k+3 \log k)$. Since every subcube $Q_d$ of 
$Q_n$ of dimension 
$d=2k + 3 \log k$ fully contains a copy of some 
$D_r(k+2 \log k, 2k+3 \log k)$ this
completes the proof of the proposition.  \hfill $\Box$

\subsection{List Erasure Codes}

In this subsection we describe the proofs of the
improved upper and lower bounds
for $L(d)$. The lower bounds follow easily from the results about
graph and hypergraph augmentations proved in subsection
\ref{ss21}. The upper bounds combine the
construction in \cite{EIL} with simple tools from linear algebra
and a computation of the Lagrangians of appropriately defined 
$t$-uniform hypergraphs.

Starting with the proof of the lower bound define, for any integer
$d \geq 2$, $g(d)$ to be the maximum possible value of the
expression
$$
\sum_{i=1}^{k+1} \prod_{j \in [k+1]-i} a_j
$$
where the maximum is taken over  all integers $k \geq 1$
and over all partitions of $d$ of the form $d=a_1+a_2+ \cdots +a_{k+1}$,
where $a_i \geq 0$ are integers. 
Thus, for example, $g(2)=2$ as demonstrated
by the partition $2=2+0$, $g(5)=8$ using the partition 
$5=2+2+1$ and $g(10)=80$ using the partition $10=2+2+2+2+2$.
\begin{lemma}
\label{l22}
For every $d \geq 2$, $L(d) \geq g(d).$
\end{lemma}
\begin{proof}
Fix a small $\eps>0$ and let $\CC$ be a collection of at least
$\eps \cdot 2^n$ vertices of $Q_n$. For a fixed $d \geq 2$ let
$g(d)= \sum_{i=1}^{k+1} \prod_{j \in [k+1]-\{i\}} a_j$
where $a_j \geq 0$ are integers. Note that this
number is exactly the number of edges of the $k$-uniform hypergraph
$H$ on $k+1$ vertex classes of sizes $a_1,a_2, \ldots, a_{k+1}$ whose
edges are all $k$-tuples containing at most $1$ vertex of each class.
(This holds even if some of the numbers $a_i$ are $0$). 
By the remark following Propositions \ref{p14} and \ref{p15}
if $n$ is sufficiently large as a function of $d$ and $\eps$,
then $\CC$ must contain $H(r)$ for some $r$. The desired
result follows as this $H(r)$ is fully contained in some
subcube of dimension $d$ in $Q_n$.
\end{proof}
The assertion of Proposition \ref{p12} follows easily from that of
the last lemma. The bounds for $L(d)$ for $5 \leq d \leq 10$ are obtained
by computing the value of $g(d)$ for these values of $d$. For large $d$
divisible by $3$, say $d=3(k+1)$, it is not difficult to check 
that the value of
$g(d)$ is obtained by the partition $a_1=a_2 = \ldots =a_{k+1}=3$,
implying that $L(d) \geq (k+1)3^k=d \cdot 3^{(d-6)/3}$. \hfill $\Box$
\vspace{0.2cm}

\noindent
We proceed with the proof of the upper bounds for $L(d)$ stated in Theorem
\ref{t13}, starting with several preliminary lemmas. 
For an integer $t \geq 1$, let $P(t)$ denote the probability that
$t$ binary vectors $v_1, v_2, \ldots ,v_t$ in $F_2^t$, each
chosen randomly, uniformly and independently among all $2^t-1$ nonzero
vectors in $F_2^t$, are linearly independent over $F_2$, that is, form
a basis of $F_2^t$. Clearly $P(1)=1$.
Choosing the vectors one by one and multiplying
the conditional probabilities that each vector is not spanned by the
previously chosen ones assuming these are linearly independent, it follows
that
\begin{equation}
	\label{e21}
	P(t)=(\frac{2^t-2}{2^t-1}) \cdot (\frac{2^t-4}{2^t-1}) \cdots
	(\frac{2^t-2^{t-1}}{2^t-1})
	= \frac{(2^t-2)(2^t-4) \cdots (2^t-2^{t-1})}{(2^t-1)^{t-1}}.
\end{equation}
It is not difficult to check that for any $t>1$
\begin{equation}
	\label{e22}
	P(t) = (\frac{2^t-2)}{2^t-1})^{t-1} P(t-1).
\end{equation}
This implies that for any $k$
\begin{equation}
	\label{e23}
	c=\lim_{t \mapsto \infty} P(t)
	= \mbox{inf}_t P(t) \geq P(k)(1-O(\frac{k}{2^k})).
\end{equation}

The equality (\ref{e22}) 
can be verified by induction on $t$, using (\ref{e21}). It 
can also be proved by the following
combinatorial argument that will be useful later too.

The nonzero vectors $v_1,v_2, \ldots ,v_t$ form a basis iff
the following two events $E_1$ and $E_2$ hold. The event $E_1$
is that 
each $v_i$ for $i \geq 2$ is not chosen to 
be equal to $v_1$. It is clear that
its probability is exactly $(\frac{2^t-2)}{2^t-1})^{t-1}$. Given the choice
of $v_1$, each nonzero vector $v$ in $F_2^t$ has a unique expression
as $v=x_v+y_v$, where
$x_v \in \{0,v_1\}$ lies in the space generated by $v_1$, and 
$y_v$ is orthogonal to this space. Let $E_2$ be the
event that the vectors $y_{v_2},y_{v_3}, \ldots ,y_{v_t}$ form 
a basis of the $(t-1)$-dimensional subspace of $F_2^t$ orthogonal
to $v_1$. Conditioning on the event $E_1$, each nonzero
vector of this $(t-1)$-dimensional space is selected with uniform
probability among these $2^{t-1}-1$ possible vectors. These vectors
span the space with probability $P(t-1)$, that is
$Prob[E_2 | E_1]=P(t-1)$. This implies (\ref{e22}) and hence
also gives that the sequence $P(t)$ is monotone decreasing
and thus approaches a limit, which is denoted by $c$ in Theorem
\ref{t13}. It is easy to check that this limit is roughly $0.29$.

We need the following simple result.
\begin{lemma}
\label{l23}
Let $t \geq 1$ and let $\{p_v: v \in F_2^t-\{0\}~ \}$ be 
an arbitrary
probability distribution on the nonzero vectors in $F_2^t$. Then
$$
	\sum_{v \in F_2^t-\{0\}} p_v(1-p_v)^{t-1} 
\leq (\frac{2^t-2}{2^t-1})^{t-1}.
$$
Equality holds for the uniform distribution $p_v=1/(2^t-1)$ for
	all $v \in F_2^t-\{0\}$.
\end{lemma}
\begin{proof}
The assertion is trivial for $t=1$. For $t \geq 2$ 
put $g(z)=z(1-z)^{t-1}$. For $t=2$ the second derivative of this
function is $-2<0$ and hence it is concave in $[0,1]$, 
implying the desired result by Jensen's inequality. For $t \geq 3$ 
the derivative and second derivative of $g(z)$ are given by
	$g'(z)= (1-z)^{t-2}(1-tz)$ and $g''(z)=(1-z)^{t-3}(-2t+2+t(t-1)z)$.
Therefore, in $[0,1]$ the function $g(z)$ is increasing in
$[0,1/t)$, attains its maximum at $z=1/t$, and is decreasing
in $[1/t,1]$. It is concave in $[0,2/t)$ and convex in
$[2/t,1]$. Suppose that the sum $\sum_v g(p_v)$ 
considered in the lemma
	attains its maximum at $(p_v: v \in F_2^t-\{0\})$  
(the maximum is clearly
attained, by compactness). If there is some $p_v>1/t$ then since
$2^t-1 > t$ there is also some $p_{v'}<1/t$. Decreasing 
$p_v$ by $\eps$ and increasing $p_{v'}$ by $\eps$, for a sufficiently
	small $\eps>0$, strictly increases 
both $g(p_v)$ and $g(p_{v'})$, contradicting maximality.
Therefore $0 \leq p_v \leq 1/t$ for all $v$. Since the function
$g(z)$ is concave in $[0,1/t]$ the maximum value of $\sum_v g(p_v)$
is obtained when all the values $p_v$ are equal, by Jensen's 
Inequality.  
\end{proof}
\begin{coro}
\label{c24}
	Let $t \geq 1$ and let $\{p_v: v \in F_2^t-\{0\}~ \}$ be 
	an arbitrary
probability distribution on the nonzero vectors in $F_2^t$. Then
the probability that a sequence $v_1,v_2, \ldots ,v_t$ of
$t$ random vectors, where each $v_i$ is
chosen randomly and independently according to this distribution,
forms a basis of $F_2^t$ is at most $P(t)$, where $P(t)$ is
defined in (\ref{e21}). This is tight and obtained by the uniform
	distribution on $F_2^t-\{0\}$.
\end{coro}
\begin{proof}
We apply induction of $t$ together with 
the reasoning described in the derivation of (\ref{e22})
from (\ref{e21}). The result is trivial for $t=1$. Assuming
it holds for $t-1$ we prove it for $t \geq 2$.
Choosing the vectors $v_1,v_2, \ldots ,v_t$
one by one, suppose $v_1=v$ (this happens with probability
$p_v$.)
The vectors $v_1,v_2, \ldots ,v_t$ form a basis iff
no $v_i$ for $i \geq 2$ is equal to $v_1=v$ (denote this event by
$E_1$), and the projections
of the vectors $v_2, \ldots ,v_t$ on the subspace orthogonal
to $v$ form a basis of this subspace (denote this event by $E_2$).
The probability that $v_1=v$ and $E_1$ holds  is $p_v(1-p_v)^{t-1}$.
The conditional probability that given this $E_2$ holds is, by 
the induction hypothesis, at most $P(t-1)$. Summing over $v$ we 
conclude that the probability that $v_1,v_2, \ldots ,v_t$ form
a basis of $F_2^t$ is at most
$$
	(\sum_{v \in F_2^v-\{0\}} p_v(1-p_v)^{t-1}) \cdot P(t-1).
$$
The first factor is at most $(\frac{2^t-2}{2^t-1})^{t-1}$, by
Lemma \ref{l23}. This and (\ref{e22}) establish the desired
inequality for $t$, completing the proof of the induction step
and of the corollary.
\end{proof}

For integers $1 \leq k \leq d$ let $B(k,d)$ denote the maximum possible
number of non-singular $k$ by $k$ submatrices in a $k$ by $d$ matrix
over $F_2$. Therefore, $B(k,d)/{d \choose k}$ is the maximum possible
probability that a set of $k$ distinct columns of such a matrix
forms a basis of $F_2^k$. 

\begin{lemma}
\label{l25}

\noindent
\begin{enumerate}
\item
For any fixed $k$ the function $B(k,d)/{d \choose k}$ is
		monotone decreasing in $d$ for all $d \geq k$, 
		and is at least $P(k)$ for every admissible $d$. 
\item
For any $1 \leq k <d$,  $B(k,d)=B(d-k,d)$.
\item
For $d \geq k^2$
$$
		(P(k) \leq )
\frac{B(k,d)}{{d \choose k}} \leq
\frac{P(k)d^k}{k! {d \choose k}} 
\leq P(k)(1+\frac{k^2}{d-k}).
$$
\item
	For any $2 \leq k \leq d$, 
	$$	
		B(k,d) \leq \lfloor \frac{d B(k-1,d-1)}{k} \rfloor.
		$$
\end{enumerate}
\end{lemma}
\begin{proof}

\noindent
\begin{enumerate}
\item
Suppose $k \leq d' <d$.
Let $A$ be a $k$ by $d$ matrix over $F_2$ which 
maximizes the probability that a random set of
$k$ of its columns forms a basis. This probability is
the average, over all choices of a $k$ by $d'$ submatrix $A'$
of $A$, of the probability that a random set of $k$ columns
of $A'$ forms a basis. The fact that
$$
		\frac{B(k,d')}{{{d '} \choose k}}
		\geq 
		\frac{B(k,d)}{{{d } \choose k}}
$$

		follows by
considering the submatrix $A'$ maximizing this
		probability. To prove the
		inequality $B(k,d)/{d \choose k} \geq P(k)$
	consider a random $k$ by $d$ matrix $A$ over $F_2$
		whose columns are chosen uniformly and independently
		in $F_2-\{0\}$. Each subset of $k$ columns of
		$A$ is a basis with probability $P(k)$ and the desired 
		inequality follows by linearity of expectation.
\item
For a $k$ by $d$ matrix $A$ of rank $k$ over $F_2$, let $A'$ denote
the $(k-d)$ by $d$ matrix whose rows form a basis of the subspace 
orthogonal to the row-space of $A$. If a set $I'$ of $(d-k)$
columns of $A'$ is of rank smaller than $d-k$ then there
is a nonzero linear combination of the rows of $A'$ 
which vanishes on these columns. This nonzero 
linear combination is orthogonal to the rows of $A$,
providing a nontrivial linear
relation of the columns $I=[d]-I'$ of $A$. This shows that if a set $I'$
of $d-k$ columns of $A'$ is not linearly independent, then the
set $I=[d]-I'$ of $k$-columns of $A$ is not linearly
independent. By symmetry the converse holds as well,
and the desired result follows by considering the matrices realizing
		$B(k,d)$
and $B(d-k,d)$.
\item
Let $A$ be a $k$ by $d$ matrix over $F_2$ with $B(k,d)$ nonsingular
$k$ by $k$ submatrices. It is clear that $A$ does not contain 
		the $0$-column (as it is easy to replace it and increase
		the number of nonsingular $k$ by $k$ submtarices).
Let $\{p_v: v \in F_2^k-\{0\}\}$
be the probability distribution assigning to each column
		of $A$ the same probability $1/d$. By Corollary \ref{c24}
if we select $k$ columns of $A$ according to this 
probability distribution (with repetition), the probability
we get a basis is at most $P(k)$. On the other hand this
probability is exactly $k!(B(k,d)/d^k$. Therefore
$ \frac{k! B(k,d)}{d^k} \leq P(k)$  implying that
$$
\frac{B(k,d)}{{d \choose k} } 
\leq \frac{P(k) d^k}{k!{d \choose k}}
\leq P(k) e^{k(k-1)/2(d-k+1)}  \leq P(k) [1+\frac{k^2}{d-k}].
$$
Here we used the fact that $\prod_{i=0}^{k-1}  (d/d-i)
< e^{k(k-1)/2(d-k+1)}$ and that $e^x \leq 1+2x$ for $x<1$.
\item
	Let $A$ be a $k$ by $d$ matrix with $B(k,d)$ $k$ by $k$ nonsingular
	submatrices. Every fixed column $c$ of $A$ can be contained in at most
	$B(k-1,d-1)$ such nonsigular matrices corresponding to the 
	$(k-1)$ by $(k-1)$ nonsingular submatrices of the $(k-1)$ by $(d-1)$
	matrix obtained from $A$ by removing $c$ and by replacing
	each column by its projection on the subspace orthogonal to $c$.
	The result thus follows by double counting.
\end{enumerate}
\end{proof}

\begin{coro}
	\label{c26}
	Put $B(d)=\max \{B(k,d): k \leq d$. Then $B(5)=5$, $B(6)=16$,
	$B(7)=28$ and $B(d) = (c+o(1)){d \choose \lfloor d/2 \rfloor}$,
	where $c$ is as in Theorem \ref{t13} and the $o(1)$-term tends to
	$0$ as $d$ tends to infinity.
\end{coro}
\begin{proof}
	By Lemma \ref{l25}, part 2, $B(d)=B(k,d)$ for some
	$k \leq d/2$. For $d=5$ it is clear that $B(1,5) \leq {5 \choose 1}
	=5<8$. $B(2,5)$ is the number of pairs of distinct columns
	of a $2$ by $5$ binary matrix in which  every column is 
	one of the three nonzero vectors of $F_2^2-\{0\}$. This is
	clearly $8$. The computation of $B(6)=B(3,6)$ and of $B(7)=
	B(3,7)$ is also  simple and is obtained by any matrix with distinct
	columns in $F_2^3-\{0\}$. (The upper bounds for these quantities
	also follow by Lemma \ref{l25}, part 4 and the fact that 
	$B(2,4)=5$).

	To estimate 
	$B(d)$ for large $d$ observe, first, that by Lemma \ref{l25}, 
	part 1 
	$$
	B(d) \geq B(\lfloor d/2 \rfloor,d) \geq
	P(\lfloor d/2 \rfloor){d \choose {\lfloor d/2 \rfloor}}>
	c{d \choose {\lfloor d/2 \rfloor}}.
	$$
Next, note 
	that for, say $k<d/4$,  $B(k,d)\leq {d \choose {d/4}}$ is (much)
	smaller than $c {d \choose {\lfloor d/2 \rfloor}}$, so 
	$B(d)=B(k,d)$ for some $d/4 \leq k \leq d/2$. By Lemma
	\ref{l25}, parts 1, 2 and 3, for any such $k$ (and
	$d \geq k+\log k$):
$$
	\frac{B(k,d)}{{d \choose k}}
	\leq \frac{B(k,k+\log k)}{{{k+ \log k} \choose k}}
	=\frac{B( \log k, k+  \log k)}{{{k+ \log k} \choose { \log k}}}
	\leq P( \log k) (1+\frac{ \log^2 k}{k})
	$$
	$$
	\leq c(1+O(\frac{ \log k}{k}))(1+\frac{ \log^2 k}{k})
	\leq c (1+O(\frac{\log^2 d}{d})),
	$$
	where in the penultimate inequality we used (\ref{e23}).

	Therefore for each $k$ in this range
	$$
	B(k,d) \leq c{d \choose k}(1+O(\frac{\log^2 d}{d})),
	$$
	completing the proof.
\end{proof}

We are now ready to prove Theorem \ref{t13}.
\begin{proof}
	The relevant erasure codes are the ones constructed in \cite{EIL},
	the novelty here is only in their analysis. 
	Here is the description of the codes 
	for a given length $n$.
	Let $C_0$ be the set consisting of the unique vector of weight
	$0$ of $Q_n$.
	For each fixed $r$, $1 \leq r \leq n$, assign to each coordinate
	$i \in [n]$ a uniformly chosen random vector $v_i$ in $F_2^r-\{0\}$.
	Let $C_r$ denote the set of all binary vectors $x$ of length $n$ and
	Hamming weight $r$ for which the $r$ vectors $v_i$ corresponding to
 all coordinates $i$ with $x_i =1$ form a basis of $F_2^r$. Note that 
	the expected cardinality of $C_r$ is ${n \choose r}P(r)
	> c { n \choose r}$, where $c$ is the limit defined in
	Theorem \ref{t13} (which is roughly $0.29$).

	Fix a choice of vectors $v_i$  for each $r$ so that the resulting
	set $C_r$ is of cardinality larger than $c { n \choose r}$
	and let $\CC$ be the union of all these sets. 
	Thus $|\CC|>c \cdot 2^n$.

	Given $d < n$, partition $\CC$ into $d+1$ pairwise disjoint
	sets, where for each $0 \leq i \leq d$ the $i$-th sets consists
	of all vectors of $\CC$ whose Hamming weight is $i \bmod (d+1)$.
	Let $\CC(d)$ denote the largest among those. Note that 
	$|\CC(d)|>(c /(d+1))2^n$ contains a constant fraction of all
	binary vectors of length $n$. 
	To complete the proof we prove an upper bound for
	the number of vectors of $\CC(d)$ in any subcube of dimension $d$
	of $Q_n$.

	Fix a subcube $D$ of dimension $d$, and let $I \subset [n]$ 
	be the set of the $d$ coordinates that vary in the subcube.
	Observe that by the choice of the Hamming weights 
	of the vectors in $\CC(d)$, $D$ can contain only
	vectors of $\CC(d)$ of one specific Hamming weight. Denote 
	this weight by $r$. Let $x$ be the common projection of all points
	of $D$ on $[n]-I$ and suppose its Hamming weight is $r-k$.
	Thus, each of the projections $y$ of
	all the vectors of $\CC(d)$ that lie in $D$ on $I$
	has weight $k$. Let $v_1,v_2, 
	\ldots ,v_{r-k}$ be the binary vectors in $F_2^r-\{0\}$ 
	that correspond to the indices $i$ where $x_i=1$. Then, by the
	construction of $C(r)$, these vectors are linearly independent. 
	Moreover, for any vector $y$ that appears as a projection above,
	the set of $k$ vectors $v_j \in F_2^r-\{0\}$ that correspond
	to the coordinates $j$ in which $y_j=1$ complete the vectors
	$v_1,v_2, \ldots ,v_{r-k}$ to a basis of $F_2^r$. This means
	that these vectors form a basis of the space of cosets of 
	$W=\mbox{span}(v_1, v_2, \ldots ,v_{r-k})$ in $F_2^r$.
	This space is isomorphic
	to $F_2^k$. It follows that the number of such projections
	$y$ is at most $B(k,d)$.

	We have thus proved that the number of vectors of $\CC(d)$ that 
	lie in $D$ is at most the maximum, over $k \leq d$, of the quantity
	$B(k,d)$, that is, at most $B(d)$. 
	The desired upper bound thus follows from the last Corollary.
\end{proof}

\subsection{Lagrangians}

Some of the discussion in the previous subsection is equivalent to
the computation of the Lagrangians of certain natural hypergraphs.
Although this is not needed for the results here, we 
briefly describe the connection
which may be of independent interest.

The Lagrangian Polynomial of a $t$-uniform hypergraph $H=(V,E)$ on
a vertex set $V=\{1,2, \ldots ,n\}$ is the polynomial
$$
P_H(x_1,x_2, \ldots ,x_n)=\sum_{e \in E} \prod_{j \in e} x_j.
$$
The Lagrangian $\lambda(H)$ of $H$ is the maximum value of 
$P_H(x_1, \ldots ,x_n)$ over the simplex $\{x_i \geq 0, \sum_i x_i=1\}$
(this maximum is attained as the simplex is compact). 

Lagrangians of hypergraphs  were first considered by Frankl and F\"uredi
\cite{FF} and by Sidorenko \cite{Si}, 
extending the application of this notion for graphs, initiated 
by Motzkin and Straus \cite{MS}.

For each $t  \geq 1$ let $B_t$ denote 
the $t$-uniform hypergraph on the vertex
set $V=F_2^t-\{0\}$ of the $2^t-1$ nonzero elements of the vector
space of dimension $t$ over $F_2$, whose  edges are all bases of 
this vector space.
Let $\lambda(B_t)$ denote the Lagrangian of this hypergraph. Trivially
$\lambda(B_1)=1$ and $\lambda(B_2)=1/3$.  By Corollary \ref{c24}
for every fixed $t$, the value of the Lagrangian
of $B_t$ satisfies $\lambda(B_t)=P(t)/t!$. Therefore
$\lambda(B_t)=(c+o(1))/t!$ where $c$ is as in Theorem \ref{t13}
and the $o(1)$-term tends to $0$ as $t$ tends to infinity.
\section{Concluding remarks}
\begin{itemize}
	\item
		By Theorem \ref{t13} $L(5)=8$.  Therefore,
		for any arbitrarily small $\eps>0$, any set of
		at least an $\eps$-fraction
		of the vertices of the $n$-cube for $n>n(\eps)$ 
		contains at least $8$ vertices in some $5$-dimensional
		subcube. On the other hand, there is  a set of
		at least $c/6> 0.04$-fraction of the vertices that does 
		not contain more than $8$ vertices of each such subcube.
		It is easy to improve this lower bound to $c/4>0.07$,
		since we can take the union of the subsets $C(r)$ for
		all Hamming weights $r$ congruent to a constant modulo
		$4$, instead of a constant modulo $6$. It is easy to check
		that this still contains at most $8$ vertices of any
		$5$-subcube, since the sum of cardinalities
		of any two quantitites $B(k,5)$ for values of $k$ 
		that differ by at least $4$ is at most $8$. Similarly, for
		large $d$ and any small $\eps>0$, the code containing all
		collections $C(r)$ for Hamming weights $r$ congruent to
		a constant modulo $b(\eps) \sqrt d$ has a fraction
		of $\Omega_{\eps}(1/\sqrt d)$ of all vertices of the
		cube and contains 
		at most
		$(c+\eps){d \choose {\lfloor d/2 \rfloor}}$
		vertices each $d$-subcube.
	\item
		As mentioned in the proof of Corollary \ref{c26} it
		is not difficult to find the exact values of 
		$B(6)=16$ and $B(7)=28$. With a bit more work one can
		determine $B(d)$ precisely for larger (small) values of
		$d$, but since there is no reason to believe that these
		provide a tight bound for $L(d)$ we have not done that.
	\item
		The problem of determining the precise value of $L(d)$
		for $d>5$ remains open. It will be interesting to
		close the gap between the upper and lower bounds for these
		quantities. Another problem is the estimation of the largest
		possible cardinality of a $(d,L(d),n)$-code. As mentioned
		above, for $d=5$ there is a $(5,8,n)$-code containing more
		than $c/4 \cdot 2^n > 0.07 \cdot 2^n$ 
		of the binary vectors of length $n$,
		but there is no reason to believe that this is tight.
		The analogous problem for $d=4$, that is, determining the
		maximum possible fraction of the set of $n$-vectors 
		in a $(4,5,n)$-code, is also
		open. The lower bound here is $c/3>0.09$ and the trivial
		upper bound is $5/16$. For smaller values of $d$ the 
		analogous problem is not difficult. The even vectors are
		$1/2$ of the vectors, and form a $(1,1,n)$-code and also
		a $(2,2,n)$-code, and $1/2$ is clearly optimal here. For
		$d=3$ the collection of all vectors of Hamming weight constant
		modulo $3$ provide a $(3,3,n)$-code with at least $1/3$ of
		all vectors. This $1/3$ is asymptotically optimal by 
		the following argument. Let $\CC$ be a collection of 
		binary vectors of length $n$. If there are more than 
		$2^{n-1}/n$ binary vectors $v$ of length $n-1$
		so that both $v0$ and $v1$ are in $\CC$, then there are
		two such vectors $v,v'$ which differ by at most $2$ 
		coordinates, and in this case $\{v0,v1,v'0,v'1\}$ all lie
		in the same $3$-cube, showing that $\CC$ is not
		a  $(3,3,n)$-code. If not, and, say, $|\CC|/2^n > (1/3+1/n)$
		then the projection of $\CC$ on the first $(n-1)$-coordinates
		is of cardinality exceeding $\lceil 2^n/3 \rceil $. By the 
		result of \cite{Ko} and \cite{JE} this projection contains
		a full $2$-cube, implying that $\CC$ contains at least
		$4$ points in a $3$-cube and showing it is not a
		$(3,3,n)$-code.
\end{itemize}
\vspace{0.2cm}

\noindent
{\bf Acknowledgment} I would like to thank Matija Buci\'c,  
Maria-Romina Ivan and Robert Johnson for helpful comments.

\end{document}